\documentclass[12pt,reqno]{amsart}
\usepackage{amssymb, amsmath, amsfonts, amsthm, graphics}
\usepackage[hmargin=1 in, vmargin = 1 in]{geometry}
\usepackage{newpxtext}
\usepackage{newpxmath}
\usepackage{tikz-cd} 
\usepackage{enumitem}
\setlist{itemsep=0.5em}
\usetikzlibrary{matrix, calc, arrows} 
\usepackage[hyphens]{url}
\usepackage{hyperref}
\usepackage{cleveref}
\usepackage[all]{xy}
\usepackage{marginnote}

\newcommand{\isom}{\cong} 

\theoremstyle{definition}

\numberwithin{equation}{section}

\DeclareMathOperator{\initial}{in}

\newcommand{\FF}{\mathbb{F}}

\newcommand{\fm}{\mathfrak{m}}

\newcommand{\s}{\mathscr{s}}

\theoremstyle{plain}
\newtheorem{thm}{Theorem}[section]
\newtheorem{Pn}[thm]{Proposition}
\newtheorem{Cor}[thm]{Corollary}
\newtheorem{lem}[thm]{Lemma}

\theoremstyle{definition}

\newtheorem{maintheorem}{Theorem}

\newcommand{\Fp}{\mathbb{F}_p}
\newcommand{\m}{\mathfrak{m}}
\newcommand{\fsig}{s}

	\title{The $F$-signature of Determinantal Rings}
 \author{Hang Huang}
	\address[H.~Huang]{Department of Mathematics, Texas A \& M University College Station, TX, 77843}
\email{\href{mailto:hhuang235@tamu.edu}{hhuang235@tamu.edu}}
\author{Cheng Meng}
\address[C.~Meng]{Yau Mathematical Sciences Center, Tsinghua University, Beijing 100084, China.}
\email{\href{mailto:cheng319000@mail.tsinghua.edu.cn}{cheng319000@mail.tsinghua.edu.cn}}
	\author{Suchitra Pande}
	\address[S.~Pande]{Department of Mathematics\\University of Utah\\Salt Lake City, 
		UT, 84112\\USA}
\email{\href{mailto:suchitra.pande@utah.edu}{suchitra.pande@utah.edu}}
\author{Yevgeniya Tarasova}
	\address[Y.~Tarasova]{Department of Mathematics\\University of Michigan\\ Ann Arbor, MI 48109\\USA}
\email{\href{ytarasov@umich.edu}{ytarasov@umich.edu}}

\begin{document}

\maketitle
\begin{abstract}
Let $R_{n,p}$ be the determinantal hypersurface ring defined by the determinant of a generic $n \times n$ matrix in characteristic $p$ where $n \geq 2$. In this paper, we obtain explicit bounds for the $F$-signature of $R_{n,p}$. For the upper bound, we prove that this $F$-signature is decreasing in $n$, so the $F$-signature of $R_{2,p}$ serves as an upper bound, which is $2/3$. For the lower bound, we find a toric degeneration $B_n$ of $R_{n,p}$, so the $F$-signature of $B_n$ serves as a lower bound. This $F$-signature can be computed using either Gr\"obner basis or Han-Monsky machinery, and its value is $\frac{(n!)^2}{(2n-1)!}$.   
\end{abstract}

 \section{Introduction}
The Frobenius endomorphism provides a means of measuring singularities in
positive characteristic. One of its numerical invariants is the
\emph{$F$-signature}, which records the asymptotic proportion of free direct
summands in iterated Frobenius pushforwards. More precisely, let $(R,\m,k)$
be a $d$-dimensional $F$-finite local domain of characteristic $p>0$ with
perfect residue field, and write
\[
 F_*^eR\cong R^{\oplus a_e(R)}\oplus M_e,
\]
where $M_e$ has no nonzero free direct summand. Then
\[
 \s(R)=\lim_{e\to\infty}\frac{a_e(R)}{p^{ed}}.
\]
Introduced by Huneke and Leuschke~\cite{HL02}, the $F$-signature exists by
Tucker's theorem~\cite{Tuc12}. It is a real number between $0$ and $1$, and equals $1$ precisely when
$R$ is regular; it is positive precisely when $R$ is strongly
$F$-regular~\cite{HL02,AL03}. Thus, the $F$-signature gives a quantitative measurement of strong $F$-regularity and moreover, captures the finiteness of fundamental groups \cite{carvajal-rojas_fundamental_2016}. Understanding these values is therefore a natural next step
beyond establishing strong $F$-regularity. Even for classical families of
singularities, however, explicit values and effective estimates are
difficult to obtain.

Generic determinantal rings form a natural family in which to investigate
this problem. Their homological and geometric properties have been studied
extensively; see~\cite{BV88}. Their strong $F$-regularity is a theorem of
Hochster and Huneke~\cite[\S7]{HH94}; a more recent proof for maximal minors
was given by Pandey and Tarasova~\cite{PT24}. These results guarantee
positivity of the $F$-signature but do not supply its value. For rings
defined by $2\times2$ minors, the identification with Segre products permits
explicit computations, as in Singh's work~\cite{Sin05}. More generally,
toric methods express $F$-signatures of normal affine semigroup rings in
terms of polyhedral data~\cite{Sin05,VK11}. For larger determinants, such
computations are substantially less accessible.

In this paper, we obtain explicit bounds for the $F$-signature of generic
determinantal hypersurfaces. For a prime $p$ and an integer $n\geq2$, set
\[
 R_{n,p}=\Fp[x_{ij}\mid 1\leq i,j\leq n]/(\det X_n),
 \qquad X_n=(x_{ij}).
\]
For a standard graded ring, we use $\fsig$ to denote the $F$-signature at its
homogeneous maximal ideal. Our main result is the following.

\begin{maintheorem}[Theorem 3.1] \label{thm:determinantal}
For every prime $p$ and every $n\geq2$,
\[
 \s(R_{n+1,p})\leq\s(R_{n,p})
 \qquad\text{and}\qquad
 \frac{(n!)^2}{(2n-1)!}\leq\s(R_{n,p})\leq\frac23.
\]
\end{maintheorem}

The inequalities hold in every positive characteristic, with bounds
depending only on the size of the matrix. For $n=2$ they recover
$\s(R_{2,p})=2/3$, while for $n=3$ they give
\[
 \frac{3}{10}\leq\s(R_{3,p})\leq\frac23.
\]
The monotonicity also describes how the $F$-signature changes as the size
of the generic determinant increases.

We also obtain an exact formula for a family of binomial hypersurfaces.
Set
\[
 B_{n,p}=\Fp[x_1,\ldots,x_n,y_1,\ldots,y_n]/
 (x_1\cdots x_n-y_1\cdots y_n).
\]
\begin{maintheorem}[Sections 5 and 6; Theorem 6.2]\label{thm:binomial}
For every prime $p$ and every $n\geq2$,
\[
 \s(B_{n,p})=\frac{(n!)^2}{(2n-1)!}.
\]
\end{maintheorem}

This formula identifies
the comparison value that gives the lower bound in
Theorem~\ref{thm:determinantal}: a simple binomial specialization yields
$\s(R_{n,p})\geq\s(B_{n,p})$. In seeking effective lower bounds, we
examined all toric degenerations of generic determinantal rings known to
us. Among the constructions considered, this simple binomial
specialization yields the strongest lower bounds we obtained for the
$F$-signatures of determinantal hypersurfaces. Related computations for
two-variable binomial hypersurface pairs appear in~\cite{BCPT26}.

We give two computations of $\s(B_{n,p})$. The first constructs an
explicit Gr\"obner basis and counts standard monomials. The second uses the Han--Monsky
representation ring~\cite{HM93}.

The paper is organized as follows: in \Cref{sec:prelim}, we review the basic facts and notations for the F-signature. In \Cref{sec: bound}, we prove our main \Cref{thm: main}, giving upper and lower bounds on the $F$-signature of determinantal hypersurfaces. We provide two different ways for the computation of the $F$-signature of those toric degenerations. In \Cref{sec: gb}  and \Cref{sec: toric F-sig}, we compute their $F$-signature via explicit Gr\"{o}ber deformations. In \Cref{sec: Han-Monsky}, we provide a different way of computing those  $F$-signature using the Han-Monsky machine.

\medskip

\paragraph{\textbf{Acknowledgements}} Most of this work was carried out at SLMath (formerly, MSRI) during the Spring 2024 semester. We thank SLMath for the institutional support and hospitality. We would also like to thank Linquan Ma, Anurag Singh, Greg Smith and Kevin Tucker for valuable discussions. Hang Huang was partially supported by the National Science Foundation under Award No. DMS-2302375.

 \section{Prelimanary} \label{sec:prelim}

\subsection{\texorpdfstring{$F$}{F}-signature and colon ideals}
Let $(R,\mathfrak m,k)$ be an $F$-finite local domain of characteristic
$p>0$ and dimension $d$. For $q=p^e$, let $F_*^eR$ denote $R$ with the
$R$-module structure $r\cdot F_*^ex=F_*^e(r^qx)$; $F$-finiteness means
that $F_*R$ is finitely generated over $R$. Write
\[
 F_*^eR\cong R^{\oplus b_e(R)}\oplus M_e,
\]
where $M_e$ has no nonzero free direct summand. Put
$a_q(R)=\frac{b_e(R)}{[k:k^q]}$. The \emph{$F$-signature} is
\begin{equation}\label{eq:prelim-fsignature}
 \s(R)=\lim_{e\to\infty}\frac{b_e(R)}{q^d[k:k^q]}
     =\lim_{e\to\infty}\frac{a_q(R)}{q^d}.
\end{equation}
It was introduced by Huneke and Leuschke~\cite{HL02}, and the existence
of the limit was proved by Tucker~\cite{Tuc12}. If $k$ is perfect, then
$[k:k^q]=1$, so $a_q(R)$ is the number of free summands of $F_*^eR$.
One has $0\leq \s(R)\leq1$, with $\s(R)=1$ when $R$ is regular
and $\s(R)>0$ when $R$ is strongly
$F$-regular~\cite{HL02,AL03}. For a positively graded ring with $R_0=k$,
$\s(R)$ denotes the $F$-signature at its homogeneous maximal ideal.

\smallskip
\noindent\textbf{Colon-ideal formulas.}
Let $(S,\mathfrak n,k)$ be an $F$-finite regular local ring and let
$R=S/J$ be a domain of dimension $d$. For an ideal $K\subseteq S$,
write $K^{[q]}=(g^q\mid g\in K)$. By
\cite[Corollary~4.14]{BST12},
\begin{equation}\label{eq:prelim-colon}
 a_q(R)=l_S\!\left(S/\bigl(\mathfrak n^{[q]}:(J^{[q]}:J)\bigr)\right),
\end{equation}
where $l_S$ denotes module length. Thus $\s(R)$ is obtained by dividing
this length by $q^d$ and taking the limit as $e\to\infty$.
For $J=(f)$, with $0\ne f\in\mathfrak n$, one has
$((f^q):(f))=(f^{q-1})$, and hence
\begin{equation}\label{eq:prelim-hypersurface}
\begin{split}
 a_q(S/(f))
 &=l_S\!\left(S/(\mathfrak n^{[q]}:f^{q-1})\right)\\
 &=q^t-l_S\!\left(S/(\mathfrak n^{[q]}+(f^{q-1}))\right),
 \qquad t=\dim S.
\end{split}
\end{equation}
The second equality follows from multiplication by $f^{q-1}$ on
$S/\mathfrak n^{[q]}$, whose length is $q^t$.


\smallskip
\noindent\textbf{Toric rings.}
For normal affine toric rings, $F$-signature can be computed as a lattice
volume. More precisely, let $N$ be a lattice with dual $M$, and let
$\sigma\subset N_{\mathbb R}$ be a full-dimensional strongly convex
rational polyhedral cone with primitive ray generators $v_1,\ldots,v_r$.
For $R=k[\sigma^\vee\cap M]$, von Korff's formula
\cite[Theorem~3.3 and Remark~6.1]{VK11} gives
\[
 s(R)=\operatorname{vol}_M
 \{u\in M_{\mathbb R}\mid0\leq\langle u,v_i\rangle<1
                              \text{ for every }i\},
\]
where $s(R)$ is taken at the maximal monomial ideal, and the measure is
normalized so that a fundamental parallelepiped of $M$ has volume one.
For a fixed cone and lattice, the value is therefore independent of the
characteristic. This gives the characteristic independence used in
Section~5. See also Singh~\cite{Sin05} for computations for affine
semigroup rings, including Segre products and Veronese subrings.

\subsection{Han-Monsky representation ring}In this subsection we fix a field $k$ of characteristic $p>0$. We introduce the concept of representation ring where we make computations. The concept of the representation ring first appears in Han and Monsky's paper \cite{HanMonskySomeSurprisingHKFunctions}, although the computational results are rooted from the results in \cite{HanHilbertKunzdiagonal}. Following \cite{HanMonskySomeSurprisingHKFunctions}, we say a $k$-object is a finitely generated $k[T]$-module annihilated by a power of $T$. Let $\Gamma^+$ be the set of isomorphic classes of $k$-objects. If $M$, $N$ are 
$k$-objects, then so is $M\oplus N$. For two $k$-objects $M,N$, we give a $k$-object structure on $M\otimes_k N$ via $T(m\otimes n)=Tm\otimes n+m\otimes Tn$. We see that $\Gamma^+$ is a commutative semiring under $\oplus$ and $\otimes$, and by the structure theorem of modules over PID, it is a free abelian semigroup over $\delta_i$ which is the class of $k[T]/(T^i),i \geq 1$; thus it is cancellative. Let $\Gamma$ be the group generated by $\Gamma^+$, then it has a commutative ring structure with addition given by $\oplus$ and multiplication given by $\otimes$. We call $\Gamma$ the \textbf{Representation ring}. We see $\Gamma$ is a free abelian group over $\delta_i$, consists of formal difference of two isomorphic classes of $k$-objects, and $\delta_1$ is the unit of $\Gamma$. For a $k$-object $M$, we denote its class in $\Gamma$ by $\gamma_M=\sum_{i \geq 1}e_{M,i}\delta_i$, where $e_{M,i}$ is the multiplicity of $k[T]/(T^i)$ in $M$. Under this notation we have $\gamma_{M\oplus N}=\gamma_M+\gamma_N$ and $\gamma_{M\otimes N}=\gamma_M\gamma_N$. For a $k$-object $M$, denote $l_M(i): \mathbb{Z} \to \mathbb{Z}$ to be the following function
\begin{equation*}
l_M(i) = \left\{
        \begin{array}{ll}
            0 & \quad i \leq 0 \\
            l(M/T^iM) & \quad i \geq 1
        \end{array}
    \right.
\end{equation*}
If $\gamma \in \Gamma$ is the class of a $k$-object $M$, we define $e_{\gamma,i}=e_{M,i}$ and $l_\gamma(i)=l_M(i)$.

\begin{Pn}\label{prop: e_M and l_M relation}
Let $M$ be a $k$-object. Then: 
\begin{enumerate}
\item \begin{equation*}
l_M(n) = \left\{
        \begin{array}{ll}
            0 & \quad n \leq 0 \\
            e_{M,1}+2e_{M,2}+\ldots+ne_{M,n}+ne_{M,n+1}+\ldots & \quad n \geq 1
        \end{array}
    \right.
    \end{equation*}
\item \begin{equation*}
l_{M}(n)-l_M(n-1) = \left\{
        \begin{array}{ll}
            0 & \quad n \leq 0\\
            e_{M,n}+e_{M,n+1}+\ldots & \quad n \geq 1
        \end{array}
    \right.
\end{equation*}
\item For any $n \geq 1$,
\begin{equation*}
e_{M,n}=l_M(n)-l_M(n-1)-(l_{M}(n+1)-l_M(n))\\
=2l_M(n)-l_M(n+1)-l_M(n-1)
\end{equation*}
\end{enumerate}
\end{Pn}
\begin{proof}
(1) follows from the definition of $l_M(n)$ and $e_{M,n}$; (2) and (3) follow from (1).
\end{proof}
We are interested in one particular case of $l_M(i)$.
\begin{Pn}\label{prop: e_Mq special case}
Let $q$ be a positive integer, $M$ be a $k$-object, and assume $T^qM=0$. Then $l_M(q-1)=l(M)-e_{M,q}$.   
\end{Pn}
\begin{proof}
By assumption $l(M)=l(M/T^qM)=l(M/T^{q+1}M)$, in other words, $l(M)=l_M(q)=l_M(q+1)$. So $e_{M,q}=2l(M)-l(M)-l_M(q-1)=l(M)-l_M(q-1)$.
\end{proof}

For $a,b,c \in \mathbb{N}$, the $k$-object $\delta_a\delta_b$ is the class of $k[T_1,T_2]/(T_1^a,T_2^b)$ as a $k$-object with respect to the action $T=T_1+T_2$. Denote $B(a,b,c)=e_{\delta_a\delta_b,c}$ and $D(a,b,c)=l_{\delta_a\delta_b}(c)$. We have:
\begin{Pn}\label{prop: delta_a times b}
For any $a,b,c \in \mathbb{Z}$, we have:
\begin{enumerate}
\item $\delta_a\delta_b=\sum_{c \geq 1} B(a,b,c)\delta_c$;
\item $D(a,b,c)=l(k[T_1,T_2]/(T_1^a,T_2^b,(T_1+T_2)^c))$;
\item For
$c \geq 1$, $B(a,b,c)=2D(a,b,c)-D(a,b,c+1)-D(a,b,c-1)$.
\end{enumerate}
\end{Pn}
\begin{proof}
(1) and (2) are clear from definition and (3) is true by (2) and \Cref{prop: e_M and l_M relation} (3).    
\end{proof}
Suppose $\gamma_M=\sum_a e_{M,a}\delta_a$ and $\gamma_N=\sum_b e_{N,b}\delta_b$, then using the bilinearity of the product, we have $\gamma_{M\otimes N}=\gamma_M\gamma_N=\sum_{a,b \geq 1}e_{M,a}e_{N,b}\delta_a\delta_b$. Thus we have:
\begin{Pn}\label{prop: e_a times b}
Let $M,N$ be two $k$-objects. Then:
\begin{enumerate}
\item $e_{M \otimes N, c}=\sum_{a,b\geq 1}e_{M,a}e_{N,b}B(a,b,c)$.
\item $l_{M \otimes N}(c)=\sum_{a,b \geq 1}e_{M,a}e_{N,b}D(a,b,c)$.
\end{enumerate}
\end{Pn}
We see from \Cref{prop: delta_a times b} and \Cref{prop: e_a times b} that the value of $B(a,b,c)$ is the key step in the computation which can be derived from the value of $D(a,b,c)$. We cite the following result from \cite{HanHilbertKunzdiagonal}.
\begin{Pn}
Let $a,b,c \in \mathbb{N}$, and $q$ be a power of $p$.
\begin{enumerate}
\item $D(a,b,c)$ is stable under any permutation of $a,b,c$.
\item If $a,b \leq q \leq c$, then $D(a,b,c)=ab$.
\item If $0 \leq a,b \leq q,c=1$. Then $D(a,b,c)=\min\{a,b\}$.
\item If $0 \leq a,b,c \leq q$, then $D(a,b,c)=D(a,q-b,q-c)+a(b+c-q)$.
\end{enumerate}
\end{Pn}
The value of $D(a,b,c)$ is easier to compute when one of $a,b,c$ is close to $0$ or $q$. We compute its value when $c=q-1$:
\begin{Cor}
If $0 \leq a,b\leq q$, then if $a+b \leq q$, $D(a,b,q-1)=ab$; if $a+b \geq q$, $D(a,b,q-1)=ab-a-b+q$.    
\end{Cor}
\begin{proof}
We see $D(a,b,q-1)=D(a,q-b,1)+(b-1)a=\min\{a,q-b\}+(b-1)a$. In the first case it is equal to $a+(b-1)a=ab$, and in the second case it is equal to $q-b+(b-1)a=ab-a-b+q$.   
\end{proof}
\begin{Cor}\label{cor: l_a times b q}
Let $a,b \in \mathbb{N}$, $q$ be a power of $p$ such that $0 \leq a,b \leq q$.
\begin{enumerate}
\item If $a+b \leq q$, then $l_{\delta_a\delta_b}(q-1)=ab$; and for any $c \geq q$, $l_{\delta_a\delta_b}(c)=ab$.
\item If $a+b \geq q$, then $l_{\delta_a\delta_b}(q-1)=ab-a-b+q$; and for any $c \geq q$, $l_{\delta_a\delta_b}(c)=ab$.
\end{enumerate}
\end{Cor}
\begin{Cor}\label{cor: e_a times b q} Let $a,b \in \mathbb{N}$, $q$ be a power of $p$ such that $0 \leq a,b \leq q$.
If $a+b \leq q$, then $e_{\delta_a\delta_b,q}=0$; if $a+b \geq q$, then $e_{\delta_a\delta_b,q}=a+b-q$.
\end{Cor}
\begin{proof}
Since $e_{\delta_a\delta_b,q}=2l_{\delta_a\delta_b}(q)-l_{\delta_a\delta_b}(q+1)-l_{\delta_a\delta_b}(q-1)$, this is true by \Cref{cor: l_a times b q}.   
\end{proof}
\begin{Cor}
Let $q$ be a power of $p$, $M,N$ be two $k$-objects such that $T^qM=0$, $T^qN=0$. Then $e_{M \otimes N,q}=\sum_{1 \leq a,b \leq q,a+b\geq q}e_{M,a}e_{N,b}(a+b-q)$.    
\end{Cor}
\begin{proof}
By \Cref{prop: e_a times b} and \Cref{cor: e_a times b q}.    
\end{proof}

\section{Bounds on F-signature of Determinantal Hypersurfaces} \label{sec: bound}

In this section, we will prove upper and lower bounds on the $F$-signature of determinantal hypersurfaces.

\begin{thm} \label{thm: main}
    Let $n \geq 2$, and let $X_{n \times n} = (x_{i,j})$ denote a generic matrix. Then, for any prime $p$, we have
    \begin{enumerate}
        \item $\s(\FF_p [X_n]/(\det(X_n))) \geq \s(\FF_p [X_{n+1}]/(\det(X_{n+1}))).$

        \item $\frac{(n!)^2}{(2n-1)!} \leq \s(\FF_p [X_n]/(\det(X_n))) \leq \frac{2}{3}.$
    \end{enumerate}
\end{thm}

\begin{proof}
    First we prove part (1): Recall that inverting the variable $x_{1,1}$ in $\FF_p [X_{n+1}]/(\det(X_{n+1}))$, we obtain the isomorphism
    \[ \FF_p [X_{n+1}]/(\det(X_{n+1}))[x_{1,1}^{-1}] \isom \FF_p [t, t^{-1}][X_{n}]/(\det(X_n)),   \]
    for a new variable $t$.
    See \cite[Proposition 2.4]{BrunsVetterDeterminantalRings} for a proof of this isomorphism. By \cite[Theorem~5.7]{PolstraUniformBoundsandSemicontinuity}, we see that $\s(\FF_p [X_{n+1}]/(\det(X_{n+1})) \leq \s(k[t, t^{-1} ][X_{n}])/(\det(X_n)) $ , where in both cases, we compute the $F$-signature at the homogeneous maximal ideals. Furthermore, since the map $\FF_p[X_{n}]/(\det(X_n)) \to \FF_p [t, t^{-1}][X_n]/( \det(X_n))$ is flat with regular fibers, we have $ \s( \FF_p[X_{n}]/(\det(X_n))) = \s( \FF_p [t, t^{-1}][X_n]/( \det(X_n)))$
    by \cite[Theorem~5.6]{YaoObservationsAboutTheFSignature}. In summary, we have shown that $ \s(\FF_p [X_n]/(\det(X_n))) \geq \s(\FF_p [X_{n+1}]/(\det(X_{n+1}))) $ as required.

    For part (2), we set $g_n  = x_{1}  \cdot \dots \cdot x_{n} + y_1 \cdot \dots \cdot y_n$. Then, we first note that when $ n =2$, the determinant is just $x_{1,1}x_{2,2}  - x_{1,2}x_{2,1}$, which after a linear change of coordinates $x_{1,2} \mapsto - x_{1,2}$ isomorphic to $g_2$. Therefore, in this case, it is sufficient to show that the $F$-signature of $g_2$ is equal to $2/3$. This is done in the calculation in \Cref{sec: toric F-sig}.
    
    When $n \geq 3$, the upper bound follows from part (1) and the $n =2$ case. For the lower bound, we note that the variables in the set $\{x_{i,j} \, | \, 1 \leq i,j \leq n \} \setminus \{x_{i,i}, x_{j, j+1}, x_{n,1} \, | \, 1 \leq i \leq n, 1 \leq j \leq n -1\} $ form a regular sequence with a quotient isomorphic to
    \[ S' = \FF_p [x_{i,i}, x_{j, j+1}, x_{n,1}]/(x_{1,1} \cdot \dots \cdot x_{n,n} + (-1) ^{n-1} x_{1,2} x_{2,3} \cdot \dots \cdot x_{n-1, n} x_{n,1}). \]
    To see this, since $\FF_p [X_{n}]/(\det(X_n))$ is a Cohen-Macaulay ring of dimension $n^2 -1$, it is sufficient to check that dimension of the quotient ring by the remaining $n^2 - 2n$ variables is exactly $2n -1$. It is easy to see from the formula for the determinant that the quotient ring is isomorphic to $S'$ which is $2n -1$ dimensional.

    So, we have show that the ring $S'$ is a quotient of $\FF_p [X_n]/(\det (X_n))$ by a regular sequence. Furthermore, note that after a linear change of cooridnates, $S'$ is isomorphic to the ring $\FF_p [x_{1}, \dots, x_{n}, y_{1} \dots, y_{n}]/(g_n)$. By the calculation in \Cref{sec: toric F-sig}, we see that the $F$-signature of $S'$ is equal to $\frac{(n!)^2}{(2n-1)!}.$ Therefore it is enough to observe that the $F$-signature drops when we quotient by a regular sequence. This is proved in \cite[Corollary 1.2]{taylor2020inversionadjunctionfsignature}. This completes the proof of the theorem, pending the calculation of the $F$-signature of $S'$.
\end{proof}
 \section{Length Computation based on Gr\"{o}bner Degeneration} \label{sec: gb}

Let $k$ be a field of characteristic $p>0$. Let $S = k[x_1,\dots,x_n,y_1,\dots,y_n]$. Let $m$ be the homogeneous maximal ideal of $R$, $f = x_1\dots x_n + y_1 \dots y_n$ and $I = m^{[p]}+(f^{p-1})$.

In this section, we will compute the length of $S/I$ by computing the length of $S/\initial_>I$ where $\initial_>I$ is the initial ideal of $I$ under a graded reverse lexicographical order generated by the standard dictionary order on the monomials. The purpose of this computation is to obtain an explicit length formula
for the $F$-signature calculation in Section~5. By the colon-ideal formula in Section~2, we have
\[
\ell\bigl(S/(\mathfrak m^{[p]}:f^{p-1})\bigr)
=
p^{2n}-\ell\bigl(S/(\mathfrak m^{[p]}+(f^{p-1}))\bigr)
=
p^{2n}-\ell(S/I).
\]
Thus, the relevant Frobenius splitting length is determined by
$\ell(S/I)$. A Gr\"obner basis for $I$ makes this length accessible:
since
\[
\ell(S/I)=\ell(S/\operatorname{in}_{>}(I)),
\]
it suffices to count the monomials outside the initial ideal.
We therefore construct an explicit Gr\"obner basis and describe the
corresponding standard monomials. The resulting count provides the
formula whose asymptotic behavior is analyzed in Section~5.

We will need the following lemmas for a Gr\"{o}bner basis of $I$. \Cref{lemma:charpIdentity} is a combinatorics fact that gives us a simple looking Gr\"{o}bner basis in \Cref{lemma:GB}. 

\begin{lem} \label{lemma:charpIdentity}
    For all $0 \leq j,v \leq p-1$, $k \geq 1$, and $j+k,v+k \leq p-1$, we have
    \[ \binom{p-1}{j} \binom{p-1}{v+k} \equiv \binom{p-1}{v} \binom{p-1}{j+k} (\bmod p) .\]
\end{lem}

\begin{proof}
    Since $(p -1)!$ is invertible mod $p$, by expanding the binomial coefficients, it is sufficient to show that
    \[ j!(p -1 - j)! (v+k)!(p -1- v - k)! \equiv v! (p -1 - v)! (j+k)! (p -1 - j - k)! (\bmod p ). \]
    But we may rewrite this (since all terms are invertible modulo $p$) as
    \[ \frac{(v+k)!}{v!} \, \frac{(p -1 - j)!}{(p -1-j - k)!} \equiv \frac{(j+k)!}{j!} \, \frac{(p -1 -v)!}{(p -1- v - k)!} \, (\bmod p) . \]
    But note that $ \frac{(p -1 - j)!}{(p -1-j - k)!} = (p -1- j) (p -1 - j - 1) \dots (p - j - k) \equiv (-1)^k \frac{(j+k)!}{j!} (\bmod p)$. And similarly, we have $ \frac{(p -1 -v)!}{(p -1- v - k)!} \equiv (- 1) ^k \frac{(v+k)!}{v!} (\bmod p)$. Therefore, we see that both sides of the above equation are equal to $ (-1)^k \frac{(v+k)!}{v!} \frac{(j+k)!}{j!} $ as required.
\end{proof}

\begin{lem} \label{lemma:GB}
Using a graded reverse lexicographical order for the standard dictionary ordering of the variables, a Gr\"{o}bner basis of $I$ consists of the following elements:
\begin{align*}
   & \{x_1^p,\dots,x_n^p\}, \\
   & \{y_1^p,\dots,y_n^p\}, \textrm{ and } \\
   & \{x_i^j \sum_{k=j}^{p-1} \binom{p-1}{k}(x_1\dots x_n)^{p-1-k}(y_1\dots y_n)^k \; \vert \; 0 \leq j \leq p-1, 1 \leq i \leq n \} .
\end{align*}
\end{lem}

 \begin{proof}

We denote the union of the sets above to be the set $G$ and we will use the following notation in this proof:

$$f_{i,j} = x_i^j \sum_{k=j}^{p-1} \binom{p-1}{k}(x_1\dots x_n)^{p-1-k}(y_1\dots y_n)^k. $$

Observe that $f_{i,0} = f^{p-1}$ for any $i$.

In order to show that the set $G$ is the Gr\"{o}bner basis, we will use Buchberger's Algorithm. We will use the following notation: $$S(f,g) = \frac{\initial_>(f)}{\gcd(\initial_>(f),\initial_>(g))}g - \frac{\initial_>(f)}{\gcd(\initial_>(g),\initial_>(g))}f$$ to represent the S-pairs. In this notation, $\initial_>(f)$ includes the lead coefficient of $f$.

We begin with the observation that $S(f_{i,j},x_i^p) = f_{i,j+1}$, thus $G$ is a generating set of $I$. Now, we will demonstrate that the rest of the S-pairs reduce to zero. 

Observe that $S(x_i^p,y_j^p) = 0$ and, for $k \neq i$, $S(f_{i,j},x_k^p) = x_i^j f_{k,j+1}$, thus both S-pairs reduce to zero. 

The S-pair $S(f_{i,j},y_k^p)$ is equal to $$y_k^{p-j}x_i^j \sum_{k=j+1}^{p-1} \binom{p-1}{k}(x_1\dots x_n)^{p-1-k}(y_1\dots y_n)^k,$$ and as every term in the polynomial is divisible by $y_k^p$, the S-pair reduces to zero. 

Now we will examine the S-pair $S(f_{i,j},f_{u,v})$, where $i \neq u$. Without loss of generality, we may assume that $j \leq v$. Then we have that 

\begin{align*}
    S(f_{i,j},f_{u,v}) = & \binom{p-1}{j}x_i^jx_u^j(x_1\dots x_n)^{v-j}\sum_{k=v+1}^{p-1} \binom{p-1}{k}(x_1\dots x_n)^{p-1-k}(y_1\dots y_n)^k \\
    & - \binom{p-1}{v}x_i^jx_u^j(y_1\dots y_n)^{v-j}\sum_{k=j+1}^{p-1} \binom{p-1}{k}(x_1\dots x_n)^{p-1-k}(y_1\dots y_n)^k\\
    =& - \binom{p-1}{v}x_i^jx_u^j(y_1\dots y_n)^{v-j}\sum_{k=p-(v-j)}^{p-1} \binom{p-1}{k}(x_1\dots x_n)^{p-1-k}(y_1\dots y_n)^k.
\end{align*}

The second equality follows from \Cref{lemma:charpIdentity}. Observe that every remaining term is divisible by $y_i^p$, thus the S-pair reduces to zero.

The proof $S(f_{i,j},f_{i,v})$ is identical, except one replaces $x_i^jx_u^j$ with $x_i^j$. Thus, we have shown all $S$-polynomials reduce to zero, and so $G$ is a Gr\"{o}bner basis. 
\end{proof}

Therefore, since $\binom{p-1}{k}$ is a unit in $k$ for $0 \leq k \leq p-1$, the initial ideal of $I$, is generated by the following set:
\[
    \{x_1^p,\dots,x_n^p,
    y_1^p,\dots,y_n^p, x_i^j (x_1\dots x_n)^{p-1-j}(y_1\dots y_n)^j \; \vert \; 0 \leq j \leq p-1, 1 \leq i \leq n \} .
\]


The following lemma computes the length of $S/I$. 

\begin{lem} \label{lengthcomputation}
The length of $S/I$ is 
\[ \ell(S/I) = p^{2n}-\Big(p^n+\sum_{k=1}^{p-1}\big((k+1)^n-2k^n+(k-1)^n\big)(p-k)^n\Big).\]
\end{lem}

\begin{proof}

Observe that the length of $S/I$ is the same as the length of $S/\initial_>(I)$. We will count the number of monomials that do not show up in $\initial_>(I)$.

Firstly, observe that any monomial where the degree of $x_i$ or $y_i$ for some $i$ is at least $p$ is already in $\initial_>(I)$. Thus, all monomials not in $\initial_>(I)$ have the form $\prod_{1 \leq i \leq n} x_i^{p-1-a_i} \prod_{1 \leq j \leq n} y_j^{p-1-b_j}$, where $0 \leq a_i,b_i \leq p-1$. There are $p^{2n}$ such monomials. 

Now, we begin to count the monomials of this form that are in $\initial_>(I)$. Observe that a monomial of the form $\prod_{1 \leq i \leq n} x_i^{p-1-a_i} \prod_{1 \leq j \leq n} y_j^{p-1-b_j}$ is in $\initial_>(I)$ if and only if the following two requirements are satisfied:
\begin{enumerate}
    \item There exists an $i$ such that $a_i = 0$,
    \item $\max\{b_1,\dots,b_n\} \leq p-1-\max\{a_1,\dots,a_n\}$.
\end{enumerate}

Based on these conditions, one may count based on the maximum of $\{a_1,\dots,a_n\}$. We begin with the case of $\max\{a_1,\dots,a_n\} = 0$. Here, we only have one choice for each $a_i$, and thus only one choice for $\prod_{1 \leq i \leq n} x_i^{p-1-a_i}$. However, in this case, we have $p$ choices for each $b_j$, thus there are $p^n$ choices for $\prod_{1 \leq j \leq n} y_j^{p-1-b_j}$, and thus $p^n$ monomials where $\max\{a_i\} = 0$.

Now, we may count what happens where $\max\{a_1,\dots,a_n\} = k$ for some $1\leq k \leq p-1$. First let's count the number of choices for $\prod_{1 \leq i \leq n} x_i^{p-1-a_i}$.

In this scenario, we must make sure to account for the fact that at least one $a_i$ must equal zero, and at least one $a_i$ must equal $k$. Observe that we have $(k+1)$ choices for each $a_i$, thus resulting in $(k+1)^n$ combinations where $\max\{a_1,\dots,a_n\} \leq k$. There are $k^n$ combinations where none of the $a_i = 0$, $k^n$ combinations where none of the $a_i = k$, and $(k-1)^n$ combinations where none of the $a_i$ are 0 or $k$. Thus, the number of choices that we have for $\prod_{1 \leq i \leq n} x_i^{p-1-a_i}$ is equal to $(k+1)^n-2k^n+(k-1)^n$.

Now we begin the count of the number of choices for $\prod_{1 \leq j \leq n} y_j^{p-1-b_j}$. Observe that all we need to do here is ensure that $\max\{b_1,\dots,b_n\} \leq p-1-k$, so we have $(p-k)$ choices for each $b_j$, and thus $(p-k)^n$ choices for $\prod_{1 \leq j \leq n} y_j^{p-1-b_j}$.

Thus, in the case where $\max\{a_1,\dots,a_n\} = k$ for $1 \leq k \leq p-1$, we have $\big((k+1)^n-2k^n+(k-1)^n\big)(p-k)^n$ monomials of the form $\prod_{1 \leq i \leq n} x_i^{p-1-a_i} \prod_{1 \leq j \leq n} y_j^{p-1-b_j}$ that are in $\initial_>(I)$.

Putting it all together, we have \[p^{2n}-\Big(p^n+\sum_{k=1}^{p-1}\big((k+1)^n-2k^n+(k-1)^n\big)(p-k)^n\Big)\] monomials of the form $\prod_{1 \leq i \leq n} x_i^{p-1-a_i} \prod_{1 \leq j \leq n} y_j^{p-1-b_j}$ which are not in $\initial_>(I)$ and thus it equals the length of $S/I$. 
\end{proof}

\section{Calculation of the toric F-signature} \label{sec: toric F-sig}

In this section, we will compute the $F$-signature of the ring $S_p/(f)$. Here $S_p = \FF_p [x_1, \dots, x_n, y_1, \dots, y_n]$ and $f $ denotes the binomial   $ x_1 \cdot \dots \cdot x_n + y_1 \cdot \dots \cdot y_n.$ Recall that the $F$-signature is the limit
\[ \s(S_p/(f)) = \lim_{e \to \infty} \frac{\ell (S_p/I_{p,e})}{p^{(2n- 1)e}}  \]
where $I_{p.e} = (\fm^{[p^e]} : (f^{p^e -1 })).$ However, since $f$ is a binomial hypersurface, $S_p/(f)$ is isomorphic to a toric ring for each $p$, and hence the $F$-signature of $S_p/(f)$ is independent of $p$. So, for any prime $p_0$ we may write
\[ \s(S_{p_0} /(f)) = \lim_{p \to \infty} \s (S_p / (f)) . \]
The reason for doing this is that by a result of P\'erez, Tucker and Yao, we have that $\lim_{p \to \infty} \s(S_p /(f)) = \lim_{p \to \infty} \frac{\ell(S_p /I_{p,1})}{p^{2n-1}}.$ Combining this with the above observation, we conclude that for any fixed prime $p_0$, we have
\[  \s(S_{p_0} /(f))  = \lim_{p \to \infty} \frac{\ell(S_p /I_{p,1})}{p^{2n-1}}. \]
Therefore, it is sufficient to compute the lengths $ \ell(S/I_{p,1})$ for all $p \gg 0$. Furthermore, we have the exact sequence
\[ 0 \to S_p /(\fm^{[p]} : (f^{p-1})) \to S/\fm^{[p]} \to S/(\fm^{[p]} + (f^{p-1})) \to 0  \]
which gives us the equality
\[ \ell(S_p/I_{p,1}) = p^{2n} - \ell (S_p/(\fm^{[p]} + (f^{p-1}))  . \]
Combining this with \Cref{lengthcomputation}, we obtain that
\begin{equation} \ell(S_p/I_{p,1}) = \Big(p^n+\sum_{k=1}^{p-1}\big((k+1)^n-2k^n+(k-1)^n\big)(p-k)^n\Big)    \end{equation}

Therefore, for any fixed $p_0$, we have
\[ \s(S_{p_0}/(f)) = \lim_{p \to \infty}  \frac{\Big(p^n+\sum_{k=1}^{p-1}\big((k+1)^n-2k^n+(k-1)^n\big)(p-k)^n\Big) }{p^{2n-1}} \]
We note that
\[ (k+1) ^n + (k - 1) ^n - 2k^n = n (n -1) k^{n-2} + o(k^{n-2}).    \]
Therefore, 
\[\lim_{p \to \infty}  \frac{\Big(p^n+\sum_{k=1}^{p-1}\big((k+1)^n-2k^n+(k-1)^n\big)(p-k)^n\Big) }{p^{2n-1}} = n ( n-1) \int_{0} ^1   x^{n-2} ( 1- x) ^n dx.   \]
This integral is the well-known $\beta$-function and the value is equal to $ \int_{0} ^1   x^{n-2} ( 1- x) ^n dx = \frac{(n-2)! \, n!}{(2n -1)!}$, which gives us the formula
\[ \s(S_p/(f)) = \frac{(n!)^2}{(2n-1)!}.    \]

\section{Han-Monsky Machine} \label{sec: Han-Monsky}

Now we use the representation ring introduced in \Cref{sec:prelim} to compute the $F$-signature of the hypersurface ring $k[x_1,x_2,\ldots,x_n,y_1,y_2,\ldots,y_n]/(x_1x_2\ldots x_n-y_1y_2\ldots y_n)$. We always assume $n \geq 2$; otherwise the quotient ring is just a polynomial ring in one variable and its $F$-signature is clearly equal to $1$. We fix a $q$ which is a power of the characteristic $p$. Let $M$ be the $k$-object
$$M=k[x_1,\ldots,x_n]/(x_1^q,\ldots,x_n^q)$$
where $T$ acts via $x_1x_2\ldots x_n$. Let $N$ be the $k$-object
$$N=k[y_1,\ldots,y_n]/(y_1^q,\ldots,y_n^q)$$
where $T$ acts via $y_1y_2\ldots y_n$. We see $M\cong N$ as $k$-objects, and $T^qM=0$, $T^qN=0$.
\begin{Pn}
Let $a$ be an integer. Then
\begin{enumerate}
\item \begin{equation*}
l_{M}(a) = \left\{
        \begin{array}{ll}
            0 & \quad a \leq 0 \\
            q^n-(q-a)^n & \quad 0<a<q\\
            q^n & \quad a \geq q
        \end{array}
    \right.
    \end{equation*}
\item \begin{equation*}
e_{M,a} = \left\{
        \begin{array}{ll}
            (q-a-1)^n+(q-a+1)^n-2(q-a)^n & \quad 0<a<q\\
            1 & \quad a = q\\
            0 & \quad a>q
        \end{array}
    \right.
    \end{equation*}
\item Let $P(x)=(x+1)^n+(x-1)^n-2x^n$ where $n \geq 2$, then $P(x)$ is a polynomial of $x$ with leading term $n(n-1)x^{n-2}$, and for $0<a<q$, $e_{M,a}=P(q-a)$.
\end{enumerate}
\end{Pn}
\begin{proof}
For part (1), it is clear when $a\leq 0$. When $0<a<q$, we have 
\begin{align*}
l(k[x_1,\ldots,x_n]/(x_1^q,\ldots,x_n^q,x_1^ax_2^a\ldots x_n^a))\\
=l(k[x_1,\ldots,x_n]/(x_1^q,\ldots,x_n^q))-l(k[x_1,\ldots,x_n]/(x_1^q,\ldots,x_n^q):x_1^ax_2^a\ldots x_n^a)\\
=l(k[x_1,\ldots,x_n]/(x_1^q,\ldots,x_n^q))-l(k[x_1,\ldots,x_n]/(x_1^{q-a},\ldots,x_n^{q-a}))\\
=q^n-(q-a)^n.
\end{align*}
When $a \geq q$, $l(k[x_1,\ldots,x_n]/(x_1^q,\ldots,x_n^q,x_1^ax_2^a\ldots x_n^a))=l(k[x_1,\ldots,x_n]/(x_1^q,\ldots,x_n^q))=q^n.$ So (1) is proved, and (2) is a consequence of (1). We have $e_{M,a}=P(q-a)$ by (2), and the rest of (3) is just computation.
\end{proof}
Since $M \cong N$ as $k$-objects, $e_{M,a}=e_{N,a}$ and $l_M(a)=l_N(a)$, so the above computation also gives $e_{N,a}$ and $l_N(a)$.
\begin{thm}
Let $a_q=l(\frac{k[x_1,x_2,\ldots,x_n,y_1,y_2,\ldots,y_n]}{(x_1^q,x_2^q,\ldots,x_n^q,y_1^q,y_2^q,\ldots,y_n^q):(x_1x_2\ldots x_n-y_1y_2\ldots y_n)^{q-1}})$. Then:
\begin{enumerate}
\item $a_q=e_{M \otimes N,q}=\sum_{1 \leq a,b \leq q,a+b\geq q}e_{M,a}e_{N,b}(a+b-q)$.
\item $\s(\frac{k[x_1,x_2,\ldots,x_n,y_1,y_2,\ldots,y_n]}{x_1x_2\ldots x_n-y_1y_2\ldots y_n})=\lim_{q \to \infty}a_q/q^{2n-1}=\frac{(n!)^2}{(2n-1)!}$.
\end{enumerate}
In particular, $\s(\frac{k[x_1,x_2,x_3,y_1,y_2,y_3]}{x_1x_2x_3-y_1y_2y_3})=3/10.$
\end{thm}
\begin{proof}
By definition of $F$-signature we see $\s(\frac{k[x_1,x_2,\ldots,x_n,y_1,y_2,\ldots,y_n]}{x_1x_2\ldots x_n-y_1y_2\ldots y_n})=\lim_{q \to \infty}a_q/q^{2n-1}$, so the last claim follows from (2). For (1), we have 
\begin{equation*}
\begin{split}
l(\frac{k[x_1,x_2,\ldots,x_n,y_1,y_2,\ldots,y_n]}{(x_1^q,x_2^q,\ldots,x_n^q,y_1^q,y_2^q,\ldots,y_n^q):(x_1x_2\ldots x_n-y_1y_2\ldots y_n)^{q-1}})\\
=l(\frac{k[x_1,x_2,\ldots,x_n,y_1,y_2,\ldots,y_n]}{(x_1^q,x_2^q,\ldots,x_n^q,y_1^q,y_2^q,\ldots,y_n^q)})\\
-l(\frac{k[x_1,x_2,\ldots,x_n,y_1,y_2,\ldots,y_n]}{(x_1^q,x_2^q,\ldots,x_n^q,y_1^q,y_2^q,\ldots,y_n^q,(x_1x_2\ldots x_n-y_1y_2\ldots y_n)^{q-1})})\\
\overset{(*)}{=}l(\frac{k[x_1,x_2,\ldots,x_n,y_1,y_2,\ldots,y_n]}{(x_1^q,x_2^q,\ldots,x_n^q,y_1^q,y_2^q,\ldots,y_n^q)})\\
-l(\frac{k[x_1,x_2,\ldots,x_n,y_1,y_2,\ldots,y_n]}{(x_1^q,x_2^q,\ldots,x_n^q,y_1^q,y_2^q,\ldots,y_n^q,(x_1x_2\ldots x_n+y_1y_2\ldots y_n)^{q-1})})\\
=l(M\otimes N)-l_{M\otimes N}(q-1)\\
=e_{M \otimes N,q}=\sum_{1 \leq a,b \leq q,a+b\geq q}e_{M,a}e_{N,b}(a+b-q).
\end{split}
\end{equation*}
Here (*) is true by applying the ring automorphism on $k[x_1,x_2,\ldots,x_n,y_1,y_2,\ldots,y_n]$ that sends $y_1$ to $-y_1$ and preserves other variables. So (1) is proved. For (2), we denote $P(x)=(x+1)^n+(x-1)^n-2x^n$. We see $P(x)=n(n-1)x^{n-2}+O(x^{n-3})$. Set
$$b_q=\sum_{1 \leq a,b \leq q,a+b\geq q}P(q-a)P(q-b)(a+b-q)$$
Assume $1\leq a,b \leq q,a+b\geq q$. Since $e_{M,a}=P(q-a)$ when $a<q$, they are not equal only when $a=q$. In this case $P(0)=1+(-1)^n,e_{M,q}=1$. Similarly $e_{N,b}=P(q-b)$ when $b<q$ and $P(0)=1+(-1)^n,e_{N,q}=1$. Thus when $q \to \infty$,
\begin{align*}
b_q-a_q=\sum_{1 \leq a,b \leq q,a=q \textup{ or }b=q}(P(q-a)P(q-b)-e_{M,a}e_{N,b})(a+b-q)\\
=\sum_{1 \leq b \leq q-1}(-1)^nbP(q-b)
+\sum_{1 \leq a \leq q-1}(-1)^naP(q-a)+q((1+(-1)^n)^2-1)
\end{align*}
Since $\deg(P)=n-2$, for any $0 \leq t \leq 1$, $\frac{\lfloor tq \rfloor\cdot P(q-\lfloor tq \rfloor)}{q^{n-1}}$ is uniformly bounded independent of $t$, and $\lim_{q \to \infty}\frac{\lfloor tq \rfloor\cdot P(q-\lfloor tq \rfloor)}{q^{n-1}}=n(n-1)t(1-t)^{n-2}$ exists. Thus
$$\frac{1}{q^n}\sum_{1 \leq b \leq q-1}bP(q-b)=\int^1_0 \frac{\lfloor tq \rfloor\cdot P(q-\lfloor tq \rfloor)}{q^{n-1}}dt \to \int^1_0 n(n-1)t(1-t)^{n-2}dt<\infty$$
So $b_q-a_q=O(q^n)=o(q^{2n-1})$ since $n \geq 2$. In particular, $\lim_{q \to \infty}a_q/q^{2n-1}$ exists if and only if $\lim_{q \to \infty}b_q/q^{2n-1}$ exists, and when they exist, they are equal. 

Now we compute $\lim_{q \to \infty}b_q/q^{2n-1}$. Note that we have for $0 \leq s,t \leq 1$,

$$\lim_{q \to \infty}\frac{P(q-\lfloor sq\rfloor)P(q-\lfloor tq \rfloor)(\lfloor sq\rfloor+\lfloor tq \rfloor-q)}{q^{2n-3}}=n^2(n-1)^2(1-s)^{n-2}(1-t)^{n-2}(s+t-1)$$
and this convergence is uniform for $s,t \in [0,1]^2$. If $s+t<1$, then for large $q$, $\lfloor sq\rfloor+\lfloor tq\rfloor<1$; if $s+t>1$, then for large $q$, $\lfloor sq\rfloor+\lfloor tq\rfloor>1$; the set $s+t=1, 0 \leq s,t \leq 1$ is a segment which has zero measure. Thus

\begin{align*}
\lim_{q \to \infty}b_q/q^{2n-1}=\lim_{q \to \infty}\sum_{1 \leq a,b \leq q,a+b\geq q}\frac{P(q-a)P(q-b)(a+b-q)}{q^{2n-1}}\\
=\lim_{q \to \infty}\int_{\Delta}\frac{P(q-\lfloor sq\rfloor)P(q-\lfloor tq \rfloor)(\lfloor sq\rfloor+\lfloor tq \rfloor-q)}{q^{2n-3}}dsdt\\
=\int_{\Delta}n^2(n-1)^2(1-s)^{n-2}(1-t)^{n-2}(s+t-1)dsdt
\end{align*}
where $\Delta=\{(s,t):0 \leq s,t \leq 1,s+t\geq 1\}$.

The rest is computation in calculus. We have
\begin{align*}
\int_{\Delta}(1-s)^{n-2}(1-t)^{n-2}(s+t-1)dsdt\\
=\int_{\Delta}-(1-s)^{n-1}(1-t)^{n-2}dsdt+\int_{\Delta}(1-s)^{n-2}t(1-t)^{n-2}dsdt\\
=\int_0^1\frac{1}{n}(1-s)^n\vert_{1-t}^{1}(1-t)^{n-2}dt+\int_0^1\frac{1}{n-1}(1-s)^{n-1}\vert_{1-t}^{1}\cdot-t(1-t)^{n-2}dt\\
=\int_0^1\frac{-t^n}{n}(1-t)^{n-2}dt+\int_0^1\frac{-t^{n-1}}{n-1}\cdot-t(1-t)^{n-2}dt\\
=\int_0^1\frac{1}{n(n-1)}t^n(1-t)^{n-2}dt
\end{align*}
So 
\begin{align*}
\lim_{q \to \infty}a_q/q^{2n-1}=\lim_{q \to \infty}b_q/q^{2n-1}=\int_0^1n(n-1)t^n(1-t)^{n-2}dt=\frac{(n!)^2}{(2n-1)!}.    
\end{align*}
\end{proof}

 \bibliographystyle{alpha}
    \bibliography{ref}
\end{document}